%% file: paper.tex
\documentclass[]{amsart}

\usepackage{amsmath}
\usepackage{amssymb}
\usepackage{amsthm}
\usepackage{graphicx}
\usepackage{epstopdf}
\usepackage{url}
\usepackage{hyperref}
\usepackage{subcaption}
\usepackage{wrapfig}
\usepackage[margin=1.5in]{geometry}

\usepackage[colorinlistoftodos,shadow]{todonotes}

\newcommand\norm[1]{\left\lVert#1\right\rVert}
\newcommand\abs[1]{\left\lvert#1\right\rvert}
\newtheorem{theorem}{Theorem}
\newtheorem{lemma}{Lemma}

\title{Efficient Evaluation of Ellipsoidal Harmonics for Potential Modeling}
\author{Thomas S. Klotz \and Jaydeep P. Bardhan \and Matthew G. Knepley}
\address[Matthew Knepley]{Department of Computational and Applied Mathematics\\
        Rice University\\
        Houston, TX 77005}
\email{knepley@rice.edu}

\address[Thomas Klotz]{Department of Computational and Applied Mathematics\\
		Rice University\\
		Houston, TX 77005}
\email{tsk1@rice.edu}

\address[Jaydeep Bardhan]{Department of Mechanical Engineering\\
		Northeastern University\\
		Boston, MA 02115}
\email{jbardhan@neu.edu}
\thanks{Research supported in part by NSF grant SI2-SSI: 1450339 and DOE Contract DE-AC02-06CH11357.}
 
 
\begin{document}

\begin{abstract}
Ellipsoidal harmonics are a useful generalization of spherical harmonics but present additional numerical challenges. One such challenge is in computing ellipsoidal normalization constants which require approximating a singular integral. In this paper, we present results for approximating normalization constants using a well-known decomposition and applying tanh-sinh quadrature to the resulting integrals. Tanh-sinh has been shown to be an effective quadrature scheme for a certain subset of singular integrands. To support our numerical results, we prove that the decomposed integrands lie in the space of functions where tanh-sinh is optimal and compare our results to a variety of similar change-of-variable quadratures.
\end{abstract}

\maketitle

%\listoftodos

\input{intro}

\input{background}

\input{theory}

\input{implementation}

\input{results}

\input{discussion}

\bibliographystyle{siam}
\bibliography{quad-paper,petsc,petscapp}

\end{document}

%% file: intro.tex
\section{Introduction}

Ellipsoidal harmonics are eigenfunctions of the Laplacian in ellipsoidal
coordinates~\cite{Byerly1893,Hobson1931,Dassios2012}, and are thus related to the spherical harmonics which are the
eigenfunctions for spherical coordinates. In fact, in ellipsoidal coordinates, the Laplace equation is separable into
three identical equations, the only coordinate system for which this is true~\cite{Miller1977}. Expansions in
ellipsoidal harmonics are suited for problems in potential theory with distributions or boundaries that are
well-modeled by ellipsoids. Unfortunately, ellipsoidal harmonics can be difficult to use, compared to their
spherical counterparts. Some of these difficulties include a lack of closed-form solutions for high-order harmonics,
normalization constants which are defined by a singular integral, multi-valued coordinate transforms, and other
implementation-specific details. For example, all computations in ellipsoidal coordinates depend on the problem-specific
choice of semi-axis lengths and must be performed again for each choice of semi-axes.

Despite computational challenges, ellipsoidal harmonics provide many computational benefits. First, if the region of
interest is better described by an ellipsoid, say a non-spherical comet~\cite{Romain2001}, then the expansion will
converge in areas much closer to the region of interest than will spherical harmonics. Thus in our example, we can
calculate the gravitational field much closer to our comet. Second, even in volumes for which both expansion converge,
the rate of convergence, as a function of order of the expansion, can be higher for ellipsoidal harmonics.

The numerics of expansion in ellipsoidal harmonics have been examined in prior work~\cite{BardhanKnepley2012a}. However,
calculation of the normalization constants for the expansion was both an accuracy and efficiency bottleneck. In this
paper, we present a method for calculating the singular integrals which defines the normalization constants with tanh-sinh
quadrature. This method is particularly suitable for high-order ellipsoidal harmonics since accuracy is only limited by
the precision of interior harmonic evaluations. In Section 2, we review the ellipsoidal coordinate system, definition of
the harmonics, and define a problem in electrostatics which we will use as an illustration. In Section 3, we define
tanh-sinh quadrature, and prove its optimality for calculation of the normalization constants. In Section 4, we detail
our software implementation, and in Section 5 we demonstrate its effectiveness for some sample electrostatic problems.

%% file: background.tex
\section{Background}

This section provides the background necessary for our implementation of ellipsoidal harmonics, but further detail on
ellipsoidal coordinates and harmonic solutions as they are used in this paper are available
in~\cite{BardhanKnepley2012a} and~\cite{Klotz2017}. Readers interested in further theory on ellipsoidal harmonics are
encouraged to consult the book of Dassios~\cite{Dassios2012} which provides a clear and comprehensive presentation.

\subsection{Ellipsoidal Coordinates}
The ellipsoidal coordinate system is a 3-dimensional generalization of the familiar 2-dimensional elliptic coordinate
system. In these coordinates, the Laplacian is fully separable. Moreover, because ellipsoidal coordinates are defined
with respect to a reference ellipsoid, it depends on the semi-axes $a$, $b$, and $c$ for definition. 

In an ellipsoidal system defined with respect to a reference ellipsoid with semi-axis lengths $a$, $b$, and $c$ given by
\begin{align}
	\frac{x^2}{a^2} + \frac{y^2}{b^2} + \frac{z^2}{c^2} = 1,
\end{align}
an ellipsoidal point $(\lambda, \mu, \nu)$ corresponding to Cartesian point $(x,y,z)$ is given by the roots of
\begin{align}\label{eq:ell2}
	\frac{x^2}{s^2} + \frac{y^2}{s^2-h^2} + \frac{z^2}{s^2-k^2} = 1
\end{align}
where $h^2 = a^2-b^2$ and $k^2 = a^2-c^2$. The squares of the roots of Eq.~\eqref{eq:ell2} are contained in the
intervals
\begin{align}
	\lambda^2 &\in (k^2,\infty) \\
	\mu^2     &\in (h^2, k^2) \\
	\nu^2     &\in (0, h^2),
\end{align}
with the sign convention
\begin{align}\label{eq:signambig}
	\text{sgn}_{\lambda} &= \text{sgn}_{x} \,\text{sgn}_{y}\, \text{sgn}_{z} \nonumber\\
	\text{sgn}_{\mu}    &= \text{sgn}_{x}\, \text{sgn}_{y} \\
	\text{sgn}_{\nu}    &= \text{sgn}_{x}\, \text{sgn}_{z}. \nonumber
\end{align}
Note that the inverse sign relation from ellipsoidal to Cartesian coordinates is given by swapping $x$ for $\lambda$,
$y$ for $\mu$, and $z$ for $\nu$ in Eq.~\ref{eq:signambig}.

\subsection{Ellipsoidal Harmonics}

Analogous to spherical harmonics, interior ellipsoidal harmonics form a basis for square-integrable functions inside the
reference ellipsoid. The Laplace equation is fully separable in ellipsoidal coordinates. If we assume a separable
solution and plug that into the Laplacian, the one-dimensional functions in each variable satisfy Lam\'e's equation
\begin{align}\label{eq:lame}
  (s^2 - h^2)(s^2 - k^2)\frac{d^2E}{ds^2}(s) + s (2 s^2 - h^2 - k^2)\frac{dE}{ds}(s) + (p - qs^2)E(s) = 0
\end{align}
where $p$ and $q$ are constants. Solutions to Lam\'e's equation fall into two classes: interior harmonics and exterior
harmonics. These solutions are used to represent functions regular on the interior and exterior of the ellipsoid
respectively. For a given order $n$, there are $2n+1$ interior harmonics each written as $E_n^p$ for $p \in [0, 2n+1)$.
These $2n+1$ solutions can be further divided into four sub-classes, $K$, $L$, $M$, and $N$ to assist in
calculation. For brevity, we will provide a brief overview of how solutions in class $K$ are derived. Further details on
the other classes are available in~\cite{BardhanKnepley2012a}.

We assume that a given $E_n^p$ in solution class $K$ can be written in the form
\begin{align}\label{eq:knpexp}
	E_n^p(s) = s^{n-2r} \sum_{j=0}^r b_j (1-\frac{s^2}{h^2})^j
\end{align}
where $r = \lfloor\frac{n}{2}\rfloor$. After plugging Eq.~\eqref{eq:knpexp} into Eq.~\eqref{eq:lame}, further
manipulation shows that the coefficients $b_j$ are given as an eigenvector of a tridiagonal matrix, which depends on the
solution class. Once these $b_j$ coefficients have been determined, the harmonic $E_n^p$ can be evaluated for different
values of $s$ using Eq.~\eqref{eq:knpexp}.

While the interior harmonics are unbounded as $\lambda \to \infty$, the exterior harmonics, $F_n^p$, are constructed
such that $F_n^p(\lambda) \to 0$ as $\lambda \to \infty$. They are defined as
\begin{align}
  F_n^p(\lambda) = (2n+1) E_n^p(\lambda) I_n^p(\lambda)
\end{align}
where
\begin{align}
  I_n^p(\lambda) = \int_{\lambda}^{\infty} \frac{ds}{[E_n^p(s)]^2 \sqrt{s^2-k^2} \sqrt{s^2-h^2}}.
\end{align}
Notice that as $\lambda \to \infty$, $I_n^p \to 0$ faster than $E_n^p \to \infty$. Thus, $F_n^p$ is guaranteed to decay
towards 0 at infinity.

The normalization constants associated with ellipsoidal harmonics are defined by the surface integral
\begin{align}
  \gamma_n^p = \int \int_{\lambda=a} (E_n^p(\mu)E_n^p(\nu))^2 \, ds
\end{align}
which can be expressed in terms of $\mu$ and $\nu$ as
\begin{align}\label{eq:normconstant}
	\int_0^h \int_h^k (E_n^p(\mu)E_n^p(\nu))^2 \frac{\mu^2-\nu^2}
	{\sqrt{(\mu^2-h^2)(k^2-\mu^2)}\sqrt{(h^2-\nu^2)(k^2-\nu^2)}}
	\, d\mu \, d\nu
\end{align}
In order to compute this integral numerically, we use the following well-known decomposition into four one-dimensional
integrals,
\begin{align}
  \gamma_n^p = 8(\mathcal{I}_1\mathcal{I}_2 - \mathcal{I}_3\mathcal{I}_4)
\end{align}
where
\begin{equation}\label{eq:normdecomp}
	\begin{aligned}
		\mathcal{I}_1 = \int_0^h \frac{(E_n^p(\nu))^2}{\sqrt{h^2-\nu^2}\sqrt{k^2-\nu^2}}\,d\nu
		\quad \quad \mathcal{I}_2 = 
		\int_h^k \frac{\mu^2(E_n^p(\mu))^2}{\sqrt{\mu^2-h^2}\sqrt{k^2-\mu^2}}\,d\mu
		\\
		\mathcal{I}_3 = 
		\int_0^h \frac{\nu^2(E_n^p(\nu))^2}{\sqrt{h^2-\nu^2}\sqrt{k^2-\nu^2}}\,d\nu
		\quad \quad
		\mathcal{I}_4 = 
		\int_h^k \frac{(E_n^p(\mu))^2}{\sqrt{\mu^2-h^2}\sqrt{k^2-\mu^2}}\, d\mu
	\end{aligned}
\end{equation}
We will directly compute these one-dimensional integrals using tanh-sinh quadrature.

\subsection{Ellipsoidal Mixed Dielectric Model}

In order to test the accuracy and efficiency of our method, we will use an ellipsoidal multipole expansion arising from
the  polarizable continuum model (PCM) of bioelectrostatics~\cite{Bardhan12_review}. To allow for an explicit
representation of the potential through an ellipsoidal harmonic expansion, we assume that the solvent is homogeneous,
infinite, and isotropic, and also that the molecular cavity is ellipsoidal.

\begin{figure}[ht]
\begin{center}
  \includegraphics{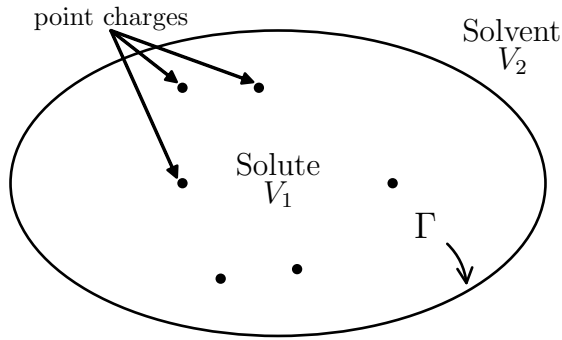}
  \caption{Mixed Dielectric Poisson Model \label{fig:ellmodel}}
\end{center}
\end{figure}
A diagram representing the model is shown in Fig.~\ref{fig:ellmodel}. The ellipsoidal boundary
$\Gamma$ separates solute cavity $V_1$ from the solvent region $V_2$. The electric permitivities, $\epsilon_1$ and
$\epsilon_2$, are assumed to be constant inside and outside of the molecular cavity and thus discontinuous at boundary
$\Gamma$. $V_1$ is also assumed to contain $Q$ point charges each at position $r_k$ with charge magnitude $q_k$ for
$k=1,\dots,Q$.

Allowing $\Phi_1$ and $\Phi_2$ to represent the electric potentials in regions $V_1$ and $V_2$ respectively, we arrive
at the following PDE:
\begin{align}
  \Delta \Phi_1(r) &= \sum_{k=1}^Q q_k \delta(r-r_k) \\
  \Delta \Phi_2(r) &= 0
\end{align}
The boundary conditions relating $\Phi_1$ and $\Phi_2$ at $\Gamma$ are the standard Maxwell boundary conditions:
continuity of potential and electric displacement. For $r_s \in \Gamma$,
\begin{align}
  \Phi_1(r_s) &= \Phi_2(r_2) \label{eq:bc1}\\
  \epsilon_1 \frac{\partial \Phi_1}{\partial n}(r_s) &= \epsilon_2 \frac{\partial \Phi_2}{\partial n}(r_s). \label{eq:bc2}
\end{align}
Finally, the electrostatic free energy due to solvation, $\Delta G^{\text{el}}_{\text{solv}}$ can be calculated directly
from the interior solution for $\Phi_1$,
\begin{equation}
  \Delta G^{\text{el}}_{\text{solv}} = \frac{1}{2} \sum_{k=1}^Q q_k \Phi_1(r_k).
\end{equation}

The physical interpretation of this model is that the solute charge distribution interacts with the dielectric medium by
inducing an apparent surface charge with density $\sigma(r)$ on the boundary of the molecular cavity. Common PCM formulations obtain
a solution for $\sigma(r)$ by discretizing over an equivalent boundary integral equation and solving the resulting
system numerically. The polarization field can then be calculated directly from $\sigma(r)$. The model above avoids
direct discretization of $\sigma(r)$ by assuming an ellipsoidal cavity, so that the polarization field can be calculated
using an expansion in ellipsoidal harmonics.

\subsection{Ellipsoidal Multipole Solution}

In order to derive an explicit multipole solution to our mixed-dielectric jump problem, we begin by assuming that the
solution on the interior can be decomposed into a sum of the Coulomb potential plus the reaction potential due to the
induced boundary charge distribution.
\begin{align}
  \Phi_1(r) = \psi^{\text{Coul}}(r) + \psi^{\text{reac}}(r)
\end{align}
We can use the well-known exterior expansion for the Green's function
\begin{align}
	\frac{1}{\norm{r-r'}} =
	\sum_{n=0}^{\infty} \sum_{p=1}^{2n+1} \frac{4\pi}{2n+1} \frac{1}{\gamma_n^p} 
	\mathbb{E}_n^p(r') \mathbb{F}_n^p(r) \label{eq:greensell}
\end{align}
which results in the following expansion for Coulomb potential valid for $\lambda \geq a$ (exterior and boundary points)
\begin{align}
	\psi^{\text{Coul}}(r) = \sum_{n=0}^{\infty} \sum_{p=1}^{2n+1} \frac{G_n^p}{\epsilon_1}\mathbb{F}_n^p(r)
\end{align}
where
\begin{align}
	G_n^p = \sum_{k=1}^Q q_k \frac{4\pi}{2n+1} \frac{1}{\gamma_n^p} \mathbb{E}_n^p (r_k).
\end{align}
If we then assume that $\psi^{\text{reac}}$ and $\Phi_2(r)$ can be expanded in terms of interior and exterior harmonics
\begin{align}
	\psi^{\text{reac}}(r) &= \sum_{n=0}^{\infty} \sum_{p=1}^{2n+1} B_n^p \mathbb{E}_n^p(r)	\label{psi} \\
	\Phi_2(r) &= \sum_{n=0}^{\infty} \sum_{p=1}^{2n+1} C_n^p \mathbb{F}_n^p(r) \label{phi2}
\end{align}
and then demand that $\Phi_1(r)$ and $\Phi_2(r)$ satisfy the boundary conditions in Eq.~\eqref{eq:bc1} and Eq.~\eqref{eq:bc2},
and the Green's function expansion in Eq.~\eqref{eq:greensell} which is valid at the boundary, recognizing that the
equation for each $n$ is independent, we arrive at
\begin{align}\label{eq:bnp}
  B_n^p = \frac{\epsilon_1 - \epsilon_2}{\epsilon_1\epsilon_2} \frac{F_n^p(a)}{E_n^p(a)} \left(1 - \frac{\epsilon_1}{\epsilon_2} \frac{\tilde{E}_n^p(a)}{\tilde{F}_n^p(a)} \right)^{-1} G_n^p
\end{align}
where
\begin{align}
  \tilde{E}_n^p (\lambda) = \frac{1}{E_n^p(\lambda)} \frac{\partial E_n^p(\lambda)}{\partial \lambda}
\end{align}
and $\tilde{F}_n^p$ is defined analogously.

%% file: theory.tex
\section{Tanh-sinh Quadrature}

Tanh-sinh quadrature is a numerical integration scheme for one-dimensional integrands on finite intervals and was
specifically designed to handle endpoint singularities~\cite{TakahasiMori1974,Mori2005}. Tanh-sinh works by first
transforming an integrand on $(-1,1)$ to $(-\infty, \infty)$ and then applying the uniform trapezoid rule to the
resulting integrand. The change of variables for $t \in (-\infty,\infty)$, is given by
\begin{align}\label{eq:tanhsinhtrans}
  \psi(t) = \tanh(\frac{\pi}{2}\sinh t), \quad \psi'(t) =\frac{\frac{\pi}{2}\cosh t}{\cosh^2(\frac{\pi}{2}\sinh t)}.
\end{align}
Applying the truncated trapezoid rule with step size $h$ then gives the formula for tanh-sinh:
\begin{gather}
  \int_{-1}^{1} f(x)\,dx \approx \sum_{k=-N}^N f(x_{kh}) w_{kh} \\
  x_{kh} = \tanh(\frac{\pi}{2} \sinh kh), \quad w_{kh} = \frac{\frac{\pi}{2} h \cosh kh}{\cosh^2 (\frac{\pi}{2}\sinh kh)}.
\end{gather}

The efficiency of this formula is a result of two properties: the decay rate of the resulting integrand after
transformation and the asymptotic optimality of the trapezoid rule for analytic integrals over $\mathbb{R}$. Due to the known
optimality of the trapezoid formula for analytic integrals over the entire real line, the optimality of tanh-sinh and other
similar quadratures depends on the efficiency of the change of variables. We begin by introducing a few function spaces
commonly used in analysis of such quadratures.

\subsection{Tanh-sinh Optimality}
First we introduce $\mathcal{D}_d$ to be a strip of finite width around the real axis in the complex plane defined by
\begin{equation}
	\phantom{.}
	\mathcal{D}_d = \{z \in \mathbb{C} \mid \abs{\Im(z)} < d\}
	.
	%\mathcal{D}^{\text{TS}}_d &= \{ \psi(z) \mid z \in \mathcal{D}_d\}
\end{equation}
We have the following theorem credited to Tanaka et al. which provides a condition to tell whether an integrand under the tanh-sinh transformation falls into the function space characterized by Sugihara~\cite{Sugihara1997} in which tanh-sinh is optimal~\cite{Tanaka2008}.
\begin{theorem}
	Assume that $f$ is analytic in $\psi(\mathcal{D}_d)$ for some $d \in (0,\pi/2)$ and suppose there exists some $C_1>0$ and $\beta>0$ such that $\forall z \in \psi(\mathcal{D}_d)$,
	\begin{equation}
		 \abs{f(z)} \leq C_1 \abs{(1-z^2)^{\beta/2-1}}. \label{eq:TSbound}
	\end{equation}
	Then there exists some C independent from N such that
	\begin{equation}
		\abs{ \int_0^1 f(x)\,dx - h\sum_{k=-N}^N f(\psi(kh)) \psi'(kh) }
		\leq
		C \exp\left(-\frac{2\pi dN}{\log(8dN/\beta)}\right)
	\end{equation}
	where
	\begin{equation}
	h = \frac{\log(8dN/\beta)}{N}
	\end{equation}
	\label{thm:tanhsinhspace}
\end{theorem}

While it is relatively straightforward check whether many integrands are bounded by $\abs{(1-z^2)^{\beta/2-1}}$ on $(-1,1)$,
it is less straightforward to show this bound holds true on $\psi(\mathcal{D}_d)$ in Eq.~\eqref{eq:TSbound}. In
Fig.~\ref{fig:D_TS}, we show a contour plot of $\psi(\mathcal{D}_d)$ for $d=\pi/4$.
\begin{figure}
	\begin{center}
	\includegraphics{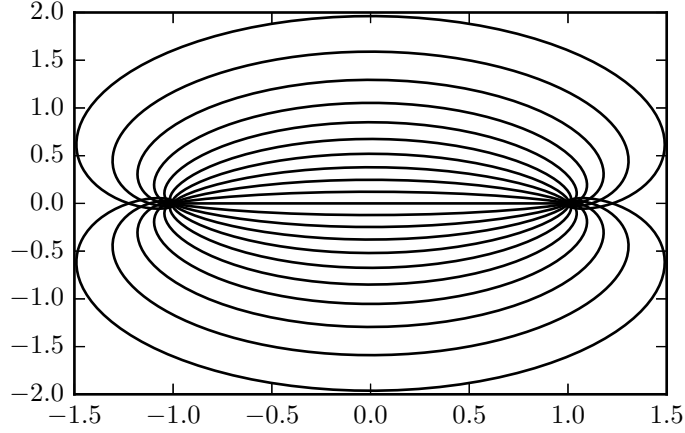}
	\end{center}
    \caption{Contours of $\psi(\mathcal{D}_d)$. Each contour corresponds to a horizontal line in $\mathcal{D}_d$ (imaginary part held constant). \label{fig:D_TS}}
\end{figure}
In the following subsection, we prove that one of the normalization constant integrands lies in this space, before
generalizing our proof in the next subsection.

\subsection{The integral $\mathcal{I}_1$}

The intuition behind this proof is to show that the height $d$ of the complex strip $D_d$ can be chosen to be small enough such that all singularities outside the region of integration are excluded from the closure of $\psi(\mathcal{D}_d)$. This guarantees that the non-singular terms are continuous and thus bounded on $\psi(\mathcal{D}_d)$, giving us the necessary conditions for Theorem~\eqref{thm:tanhsinhspace}.

We start by looking at $\mathcal{I}_1$
\begin{equation}
	\mathcal{I}_1 = \int_0^h \frac{(E_n^p(\nu))^2}{\sqrt{h^2-\nu^2}\sqrt{k^2-\nu^2}}\,d\nu
\end{equation}
and transform this integral onto $(-1,1)$ using the change of variables $z = \frac{2}{h}\nu - 1$ resulting in
\begin{equation}
	\phantom{.}\mathcal{I}_1 =
	\int_{-1}^{1}  \frac{E_n^p(\frac{h}{2}(z+1))}
	{\sqrt{1-\frac{(z+1)^2}{4}}\sqrt{k^2-\frac{h^2(z+1)^2}{4}}}
	 \frac{h}{2}\, dz.
\end{equation}
The transformed integrand is continuous on $\mathbb{C}$ with the exception of $z=1$, $z=-3$, $z=\frac{2k}{h}-1$, and $z=-\frac{2k}{h}-1$. Thus, we want to choose $d$ such that the latter three singularities are excluded from $\overline{\psi(\mathcal{D}_d)}$ so that the conditions for Theorem~\eqref{thm:tanhsinhspace} are met.

Because $\psi$ is Lipschitz for $d<\frac{\pi}{2}$, showing that the latter three singularities are excluded from $\overline{\psi(\mathcal{D}_d)}$ can be done by showing that the points in $\mathbb{C}$ which map to singularities under $\psi$ must be bounded away from the real axis, and thus can be excluded from $\mathcal{D}_d$. The following describes how this can be done by contradiction.

If we take $\psi$ to be the tanh-sinh transformation and look at the set of points mapping to the latter three singularities,
\begin{equation}
	S = \left\{ z\in\mathbb{C} \mid \psi(z) = -3 \text{ or } \psi(z) = \frac{2k}{h}-1 \text{ or } \psi(z) = -\frac{2k}{h}-1 \right\},
\end{equation}
then we can show through contradiction that these must be bounded away from the real axis. Assuming for contradiction that they aren't bounded away from the real axis, there exists a sequence of points $(x_j)_{j=1}^\infty$ where each $x_j \in S$ and $|x_j - \Re(x_j)| \to 0$ as $j \to \infty$. However, because $\psi$ is Lipschitz, we know that $|\psi(x_j) - \psi(\Re(x_j))| \to 0$ as $j \to \infty$ as well. However, this is a contradiction because each $\psi(\Re(x_j)) \in (-1,1)$ and $\psi(x_j)$ is a singular point of the integrand, implying that there is a sequence of singular points getting arbitrarily close to $(-1,1)$ which is impossible with only three singular points. Therefore, if we choose $d = \inf(\Im(S))/2$, then $d>0$ and by our choice of $d$, $\overline{\mathcal{D}_d} \cap S = \emptyset$.

Because $\overline{\mathcal{D}_d} \cap S = \emptyset$, the non-singular portion of our integrand contains no singularities and must be continuous on $\overline{\psi(\mathcal{D}_d)}$. Finally, because $\overline{\psi(\mathcal{D}_d)}$ is compact, this portion of our integrand is bounded by a constant resulting in the necessary conditions for Theorem~\eqref{thm:tanhsinhspace}.

Notice that this result doesn't rely on any specific properties of the non-singular portion of the integrand other than assuming it is analytic with a finite number of complex singularities. Thus, we can generalize this proof to functions of this type as is demonstrating in the following section.

\subsection{General proof}
For the following section,
\begin{equation}
	\psi(z) = \tanh\left(\frac{\pi}{2}\sinh(z)\right)
\end{equation}
\begin{lemma}
	$\psi(z)$ is Lipschitz on $\mathcal{D}_d$ whenever $d < \frac{\pi}{2}$
\end{lemma}
A proof of this algebraically is given in \cite{Tanaka2008}, but we provide a short proof based on Cauchy's derivative estimate. We present this proof because it only relies on the boundedness of $\psi$ on $\mathcal{D}_d$ and thus could easily be extended to other transformations.
\begin{proof}
	First, we fix $d$ arbitrarily such that $0 < d < \frac{\pi}{2}$. Then, given that $\psi$ is bounded on $\mathcal{D}_d$ for any $d < \frac{\pi}{2}$, there exists an $M$ such that $\abs{\psi(x)} \leq M$ on $\mathcal{D}_d$.
	
	Cauchy's derivative estimate, a direct consequence of Cauchy's integral formula, states that if $\abs{f(x)} \leq M$ on $C$, a circle of radius r centered at a, then
	\begin{equation}
		\abs{f^{(n)}(a)} \leq \frac{n!M}{r^n}
	\end{equation}
	
	Choosing an $h$ such that $d < h < \frac{\pi}{2}$ and considering any $z \in \mathcal{D}_d$, we examine a circle $C_z$ of radius $h-d$ centered at the point $z$. For any $y \in C_z$, we can apply the identity
	\begin{equation}
		\Im(y) = \Im(y-z) + \Im(z)
	\end{equation}
	and because $\abs{\Im(y-z)} \leq \abs{y-z} \leq h-d$ and $z \in D_d$, we have that
	\begin{align}
		\abs{\Im(y)} &\leq \abs{\Im(y-z)} + \abs{\Im(z)} \\
		\abs{\Im(y)} &< h
	\end{align}
	and therefore $C_z \subset \mathcal{D}_h$. By our original assumption we know that there exists some $M_h$ such that $\abs{\psi(x)} \leq M_h$ on $\mathcal{D}_h$. We have now shown that $\psi$ satisfies the conditions for Cauchy's derivative estimate on the circle $C_z$ for any $z \in \mathcal{D}_d$. Therefore, we have
	\begin{equation}
		\abs{\psi'(z)} \leq \frac{M_h}{h-d}
		\label{eq:cauchybound}
	\end{equation}
	in $C_z$. Furthermore, we observe that the right hand side of equation (\ref{eq:cauchybound}) is independent of $z$. This proves that $\psi$ has bounded derivative on $\mathcal{D}_d$ and as a consequence $\psi$ is also Lipschitz on $\mathcal{D}_d$.
\end{proof}

\begin{lemma}
	$\psi(\mathcal{D}_d)$ is bounded. \label{lem:tsregionbounded}
\end{lemma}
\begin{proof}
	Consider any point $z \in \psi(\mathcal{D}_d)$. By construction of this set, there exists some $x \in \mathcal{D}_d$ where $\psi(x) = z$. Furthermore, because $\psi$ is Lipschitz, we have that $\abs{\psi(x) - \psi(\Re(x))} = \abs{z-\psi(\Re(x))} \leq L \abs{x-\Re(x)} = L\abs{\Im(x)} \leq Ld$. Thus,
	\begin{equation}
		\abs{z-\psi(\Re(x))} \leq Ld
	\end{equation}
	and $\psi(\Re(x)) \in (-1,1)$, guaranteeing that any z arbitrarily chosen in $\psi(\mathcal{D}_d)$ is within a finite distance of $(-1,1)$ independent of the choice of $z$. Therefore, $\psi(\mathcal{D}_d)$ is bounded.
\end{proof}

The following theorem is our general result 

\begin{theorem}
	Consider $d \in (0,\frac{\pi}{2})$ and a function $f$ of the form
	\begin{equation}
		f(x) = \frac{g(x)}{h(x)}
	\end{equation}
	where $f,g$ are analytic, the function $h$ satisfies 
	\begin{equation}
		\abs{h(x)} \geq C_1 \abs{\sqrt{1-x^2}}
		\label{eq:hbounded}
	\end{equation}
	in $\psi(\mathcal{D}_d)$, and the function $g$ is continuous on $[-1,1]$ with a finite number of singularities in the complex plane.
	
	Then $f$ satisfies the conditions for Theorem~\eqref{thm:tanhsinhspace}.
\end{theorem}
\begin{proof}
	We define S to be the points in $\mathbb{C}$ which map to singularities of $g$ under $\psi$ i.e.,
	\begin{equation}
		\phantom{.}
		S = \{ z \in \mathbb{C} \mid g \text{ is singular at } \psi(z) \}
		.
	\end{equation}

We contend that the points in $S$ are bounded away from the real axis. To show this, assume that they are not. Then, there exists a sequence of points $(z_k)_{k=1}^{\infty}$ where $z_k \in S$ and $\abs{z_k - \Re(z_k)} \to 0$. Because $\psi$ is Lipschitz, we have that $\abs{\psi(z_k) - \psi(\Re(z_k))} \leq L \abs{z_k - \Re(z_k)}$ for some constant $L$. From this inequality we can conclude that $\abs{\psi(z_k) - \psi(\Re(z_k))} \to 0$ as well which is an infinite sequence of singular points of $g$ getting arbitrarily close to points in the set $(0,1)$. However, our original assumption that $g$ is continuous on $[-1,1]$ guarantees that $\psi(z_k)$ is not a limit point of $(-1,1)$, allowing us to extract an infinite subsequence of unique points $(\psi(z_{n_k}))_{k=1}^{\infty}$ getting monotonically closer to the set $(-1,1)$. The fact that this subsequence contains unique points contradicts our original assumption that $g$ has finite singularities in the complex plane.

Therefore, the points in $S$ are bounded away from the real axis i.e.,
\begin{equation}
	\phantom{.}
	\inf_{x \in S}{|\Im(x)|} > 0
	.
\end{equation}
If we choose $d$ to be $d:=\frac{1}{2}\inf_{x \in S}{\abs{\Im(x)}}$, then $\mathcal{D}_d \cap S = \emptyset$ and $\overline{\mathcal{D}_d} \cap S = \emptyset$. Furthermore, $\overline{\psi(\mathcal{D}_d)} \cap \psi(S) = \emptyset$ and due to Lemma~\eqref{lem:tsregionbounded}, $g$ must be bounded on the compact set $\overline{\psi(\mathcal{D}_d)}$ by some constant $C_2$. This, combined with our original assumption in equation (\ref{eq:hbounded}) that $h(x)$ is bounded below results in
\begin{equation}
	\abs{f(x)} = \frac{\abs{g(x)}}{\abs{h(x)}} \leq \frac{C_2}{C_1} \abs{\frac{1}{\sqrt{1-x^2}}}
\end{equation}
as desired.
\end{proof}

%Now considering $\psi(\mathcal{D}_d)$, the closure of this set can't contain $-2$, $\frac{2k}{h}-1$ or $-\frac{2k}{h}-1$ or else the imaginary part of points in $S$ would not be bounded away from $d$ by continuity of $\psi$. Additionally, the closure of $\psi(\mathcal{D}_d)$ is bounded because $\psi(\mathcal{D}_d)$ is bounded. This gives us that the function \tom{This is all a little messy. Will clean up later}
%\begin{equation}
%	\frac{E_n^p(\frac{h}{2}(z+1))}{\sqrt{k^2-\frac{h^2(z+1)^2}{4}}} \frac{h}{2}
%\end{equation}
%is continuous and bounded on the closure of $\psi(\mathcal{D}_d)$ and thus also bounded on $\psi(\mathcal{D}_d)$. So we can say $\forall z \in D_d^\text{TS},$
%\begin{equation}
%	\left|\frac{E_n^p(\frac{h}{2}(z+1))}
%	{\sqrt{1-\frac{(z+1)^2}{4}}\sqrt{k^2-\frac{h^2(z+1)^2}{4}}}
%	\frac{h}{2} \right|
%	\leq
%	C_1 \left|\frac{1}{\sqrt{1-\frac{(z+1)^2}{4}}}\right|
%\end{equation}
%for some $C_1 > 0$. Trivially, it can be shown that for another $C_2>0$,
%\begin{equation}
%	\left|\frac{1}{\sqrt{1-\frac{(z+1)^2}{4}}}\right| \leq C_2 \left|\frac{1}{\sqrt{1-z^2}}\right|
%\end{equation}
%which satisfies the conditions for the theorem.

%% file: implementation.tex
\section{Implementation}

Our ellipsoidal implementation follows~\cite{BardhanKnepley2012a} closely, but reimplements the functionality in C
whereas the original used Python and MATLAB. In addition, we have used
MPFR package~\cite{FousseHanrotLefevrePelissierZimmermann2007} for arbitrary precision calculations. Our work has been
integrated into the PETSc package~\cite{petsc-user-ref,petsc-web-page} from Argonne National Laboratory. The user can
select either the fixed or arbitrary precision versions for a given integral. Dense linear algebra for the fixed precision version, such as the
eigensolves used for harmonic initialization, were performed using LAPACK~\cite{Anderson1999}. In the arbitrary precision version, linear algebra routines were implemented in MPFR and the eigenvalue problem was solved using a power iteration with reorthogonalization.

Coordinate transforms between the ellipsoidal and Cartesian systems were originally presented in~\cite{Romain2001} and
further detailed in~\cite{BardhanKnepley2012a}, which also presented a method for handling sign ambiguities (our
Eq.~\eqref{eq:signambig}). We also follow~\cite{BardhanKnepley2011} for the spherical harmonic implementation.

When initializing an ellipsoidal system and associated harmonics, abscissas ($x_{kh}$) and weights ($w_{kh}$) are
precomputed for $k=-N,\dots,N$ where $N$ is chosen so that $w_{Nh} \leq 10^{-2p}$ where $p$ is the desired digits of
precision, as recommended by Bailey~\cite{Bailey2005}. The step size $h$ used for initialization is $h_j = 2^{-j}$
where $j \in \mathbb{N}$ is provided by the user. Due to nested definition of tanh-sinh abscissas, this precomputation
includes all abscissas computed with step size $h_i$ for $0\leq i<j$ up to the same $10^{-2p}$ tolerance for the
weights. 

The computational complexity associated with this initialization is $O(N)$ which contrasts Gaussian quadrature's
$O(N^2)$ initialization cost using common approaches, although we do note the recent development of an $O(N)$
initialization scheme for Gaussian quadrature~\cite{GlaserLiuRokhlin2007} albeit with larger
constant. After initialization, the computational complexity associated with evaluating and accumulating $N$
abscissa-weight pairs is $O(N)$ and the error decreases at a rate of $O\left(\exp(-c_1 \frac{N}{\log N})\right)$.
After the abscissa/weight initialization, the inner trapezoid rule computation in tanh-sinh is implemented adaptively
with the step size halved at each level. Convergence was assessed using the estimates provided by
Bailey~\cite{Bailey2005} which were uniformly close to the actual errors in tests.

%% file: results.tex
\section{Results}

\subsection{Implementation Accuracy}

In order to verify our computation of the normalization constants, we will compare to both explicit formulae (for degree
up to two) and to a high precision MPFR implementation executed with approximately 1000-digit arithmetic. As implemented,
normalization constants of an arbitrary order can be independently computed with accuracy only limited by the precision
of interior harmonic evaluation. As demonstrated in Fig.~\ref{fig:normConv}, the double precision calculation is
accurate for low-order harmonics, but the accuracy degrades for high-order harmonics.

\begin{figure}[h]
	\makebox[\textwidth][c]{
		\includegraphics[]{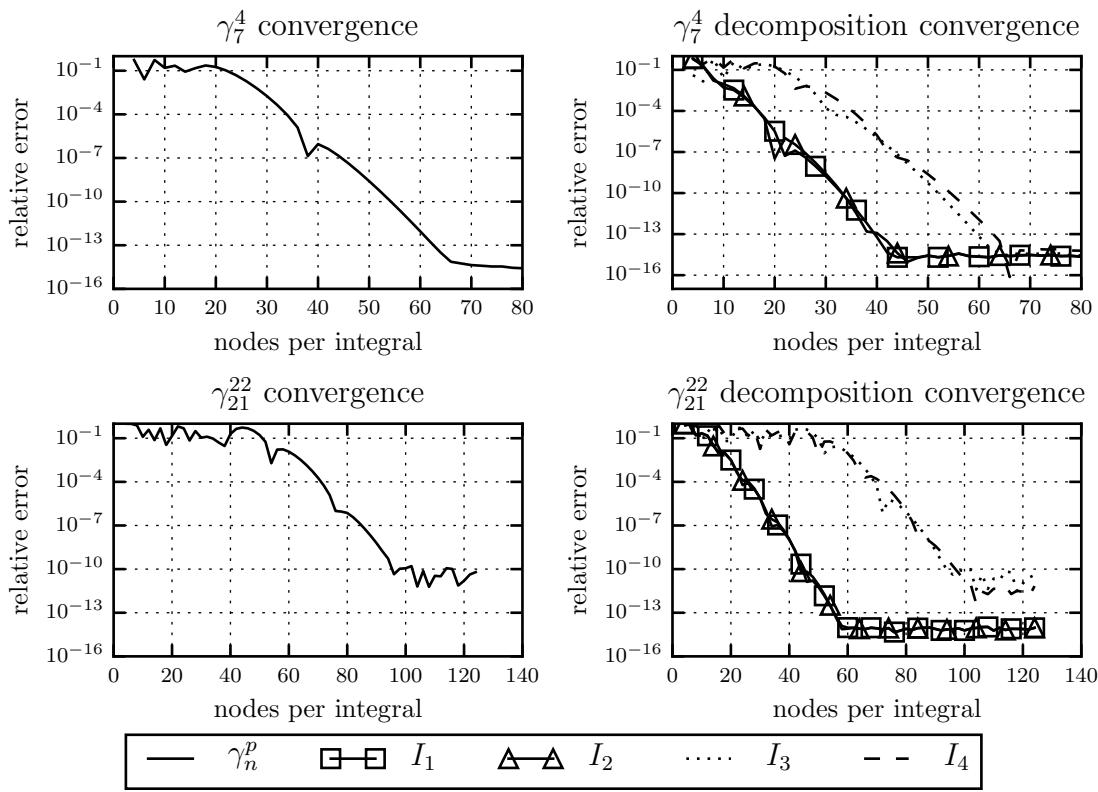}
	}
	\caption{Convergence of normalization constants for ellipsoid with a=3, b=2, c=1.}
	\label{fig:normConv}
\end{figure}

In order to examine the efficiency of the tanh-sinh transformation, we compare it to two popular transformations: error
function and tanh. Because quadrature nodes and weights can be precomputed for many applications, we choose to compare
accuracy against the number of function evaluations. In Fig.~\ref{fig:pointpreccomp}, we see that tanh-sinh is not only
more efficient, but has a higher rate of convergence.

\begin{figure}[h]
  \begin{center}
    \includegraphics{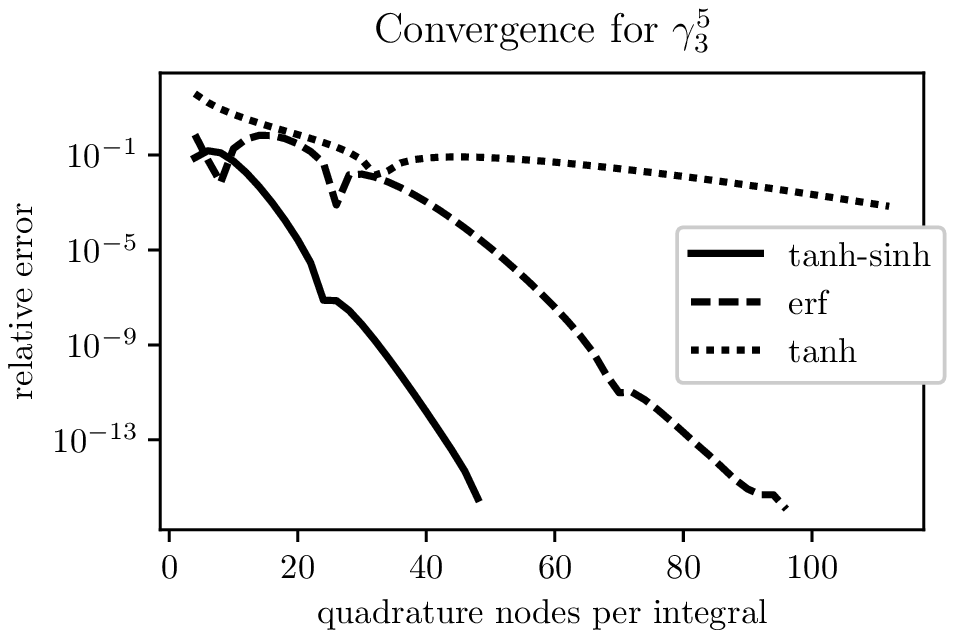}
  \end{center}
  \caption{Comparison of three integral transforms for the computation of a fifth order normalization constant. \label{fig:pointpreccomp}}
\end{figure}

\subsection{Ellipsoidal Potential Problems}

The so-called Born ion, a single charge at the center of a dielectric sphere, provides a verification of our overall
energy calculation. We construct a nearly spherical cavity and look at the limit as its becomes a sphere. In
Fig.~\ref{fig:sphConv}, we consider a single unit point charge at the origin and cavity semi-axes of lengths
$a=1+\Delta$, $b=1+\Delta/5$, $c=1+\Delta/10$ as $\Delta \to 0$. We see that the free energy converges linearly to that
of the Born ion.
%\begin{wrapfigure}[16]{R}{0.5\textwidth}
\begin{figure}[h]
  \begin{center}
    \includegraphics[width=.7\textwidth]{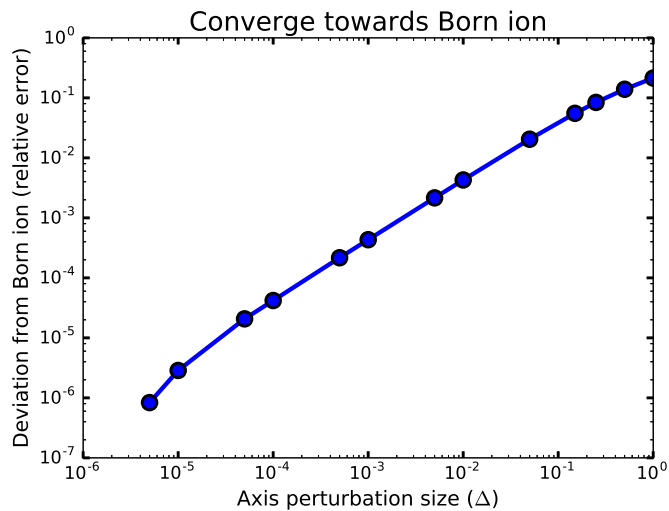}
  \end{center}
  \caption{Error in free energy against the Born ion as cavity becomes spherical. \label{fig:sphConv}}
\end{figure}
%\end{wrapfigure}

We would also like to examine the convergence of our expansion for the simple problem of an ellipsoidal cavity with semi-axis
lengths $a=3$, $b=2$, $c=1$ and five randomly placed interior charges. As shown in Fig.~\ref{fig:cavityConv}, the
solvation free energy, ($\Delta G^{\text{el}}_{\text{solv}}$), converges exponentially with respect to the maximum order
of expansion, as expected. In the work-precision diagram, we see that convergence is sublinear with respect to flops due
to the nonlinear relationship between work required and order of expansion. However, it is worth mentioning that most
calculations performed to generate this solution can be trivially parallelized.

\begin{figure}[h]
  \begin{center}
    \includegraphics{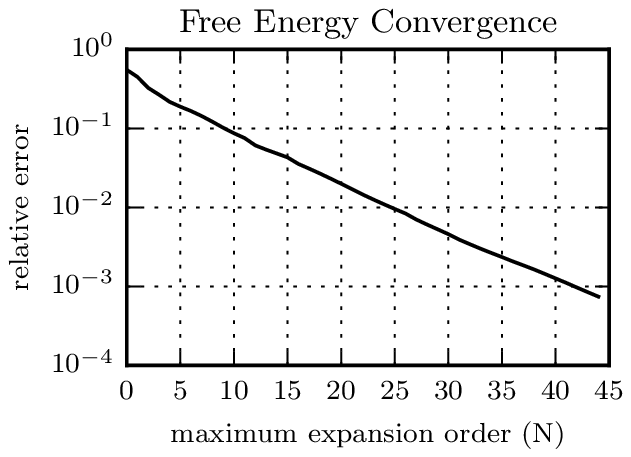}
    \includegraphics{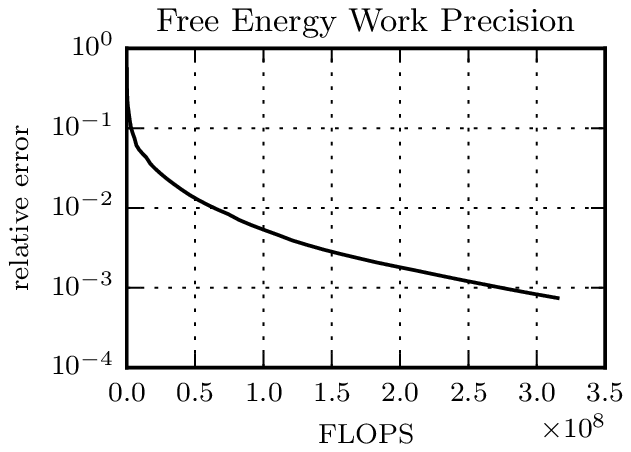}
  \end{center}
  \caption{Convergence and work-precision diagrams for the free energy of an ellipsoidal cavity with 5 randomly placed charges. \label{fig:cavityConv}}
\end{figure}

For approximately ellipsoidal charge distributions, the ellipsoidal expansion can have a substantially larger region of
convergence. For the same ellipse and charge configuration as above, we show the convergence of spherical and
ellipsoidal expansions in Fig.~\ref{fig:sphElipComp}. In order to compare solutions, we only examine the Coulomb portion
of each potential. In the left figure, we see that the ellipsoidal expansion convergences for points outside the
containing ellipse, but inside the Brillouin Sphere, which is the smallest sphere containing the charges.

\begin{figure}[h]
  \begin{center}
    \includegraphics{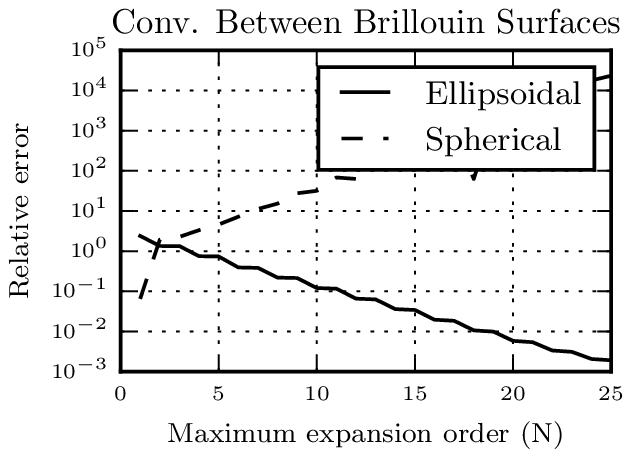} 
    \includegraphics{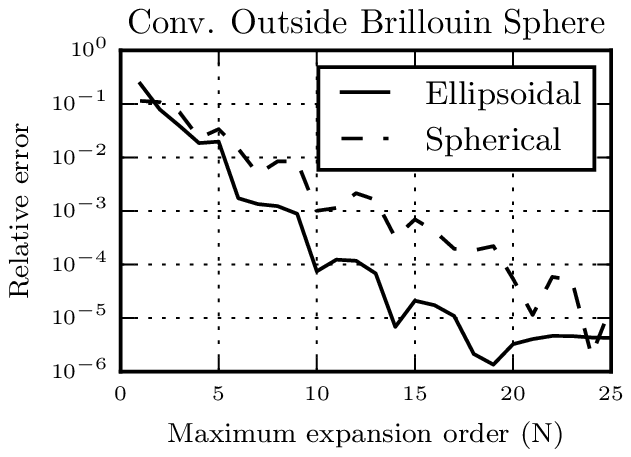}
  \end{center}
  \caption{Convergence of the Coulomb potential for spherical and ellipsoidal expansions for the potential due to an
    ellipsoidal distribution of charge, at a point inside the sphere containing the charge but outside the ellipsoid
    (left) and a point outside both the sphere and ellipsoid (right). \label{fig:sphElipComp}}
\end{figure}

%% file: discussion.tex
\section{Discussion}

We have provided simple implementation capable of accurately calculating ellipsoidal harmonic expansions at least to
order 50. We make use of tanh-sinh quadrature, taking advantage of adaptivity in the algorithm, which may be easily
coded by the user, but we also make available a free C implementation in the PETSc
source~\cite{petsc-web-page,petsc-user-ref}. We have demonstrated a superior convergence rate from ellipsoidal
harmonics, and more robust convergence when sources are better enclosed by an ellipsoid. We hope this enables spectral
methods based on ellipsoidal harmonics as a viable alternative to the more common spherical harmonic expansion.

More broadly, it seems that our approach to optimality of the quadrature can be applied to other classes of singular
integrals, for examples those that make up the Galerkin boundary element method. Currently, analytic integrals or
quadratures specialized to the kernel or geometry are often used, but this seems to provide an automatable method for
attacking the problem, which might be complementary to the QBX approach~\cite{KlocknerBarnettGreengardONeil2013}. We
also believe that the spectral method for the PCM using ellipsoidal harmonics could be used as a valuable supplement to
operator approximation methods for the solvation problem~\cite{BardhanKnepleyAnitescu2008,BardhanKnepley2011}.

%% file: paper.bbl
\begin{thebibliography}{10}

\bibitem{Anderson1999}
{\scshape E.~Anderson, Z.~Bai, C.~Bischof, L.~S. Blackford, J.~Demmel, J.~J.
  Dongarra, J.~Du~Croz, S.~Hammarling, A.~Greenbaum, A.~McKenney, and
  D.~Sorensen}, {\em LAPACK Users' Guide (Third Ed.)}, Society for Industrial
  and Applied Mathematics, Philadelphia, PA, USA, 1999.

\bibitem{Bailey2005}
{\scshape D.~H. Bailey, K.~Jeyabalan, and X.~S. Li}, {\em A comparison of three
  high-precision quadrature schemes}, Experimental Mathematics, 14 (2005),
  pp.~317--329.

\bibitem{petsc-user-ref}
{\scshape S.~Balay, S.~Abhyankar, M.~F. Adams, J.~Brown, P.~Brune,
  K.~Buschelman, L.~Dalcin, V.~Eijkhout, W.~D. Gropp, D.~Kaushik, M.~G.
  Knepley, L.~C. McInnes, T.~Munson, K.~Rupp, B.~F. Smith, S.~Zampini,
  H.~Zhang, and H.~Zhang}, {\em {PETS}c users manual}, Tech. Rep. ANL-95/11 -
  Revision 3.7, Argonne National Laboratory, 2016.

\bibitem{petsc-web-page}
\leavevmode\vrule height 2pt depth -1.6pt width 23pt, {\em {PETS}c {W}eb page}.
\newblock \url{http://www.mcs.anl.gov/petsc}, 2017.

\bibitem{Bardhan12_review}
{\scshape J.~P. Bardhan}, {\em Biomolecular electrostatics---{I} want your
  solvation (model)}, Computational Science \& Discovery, 5 (2012), p.~013001.

\bibitem{BardhanKnepley2011}
{\scshape J.~P. Bardhan and M.~G. Knepley}, {\em Mathematical analysis of the
  {BIBEE} approximation for molecular solvation: Exact results for spherical
  inclusions}, Journal of Chemical Physics, 135 (2011), pp.~124107--124117.
\newblock \url{http://arxiv.org/abs/1109.0651}.

\bibitem{BardhanKnepley2012a}
\leavevmode\vrule height 2pt depth -1.6pt width 23pt, {\em Computational
  science and re-discovery: open-source implementations of ellipsoidal
  harmonics for problems in potential theory}, Computational Science \&
  Discovery, 5 (2012), p.~014006.
\newblock \url{http://arxiv.org/abs/1204.0267}.

\bibitem{BardhanKnepleyAnitescu2008}
{\scshape J.~P. Bardhan, M.~G. Knepley, and M.~Anitescu}, {\em Bounding the
  electrostatic free energies associated with linear continuum models of
  molecular solvation}, Journal of Chemical Physics, 130 (2008), p.~104108.
\newblock Selected for the March 15, 2009 issue of Virtual Journal of
  Biological Physics Research, \url{http://dx.doi.org/10.1063/1.3081148}.

\bibitem{Byerly1893}
{\scshape W.~E. Byerly}, {\em An elementary treatise on Fourier's series: and
  spherical, cylindrical, and ellipsoidal harmonics, with applications to
  problems in mathematical physics}, Dover, 1893.

\bibitem{Dassios2012}
{\scshape G.~Dassios}, {\em Ellipsoidal harmonics: theory and applications},
  vol.~146, Cambridge University Press, 2012.

\bibitem{FousseHanrotLefevrePelissierZimmermann2007}
{\scshape L.~Fousse, G.~Hanrot, V.~Lef{\`e}vre, P.~P{\'e}lissier, and
  P.~Zimmermann}, {\em {MPFR}: A multiple-precision binary floating-point
  library with correct rounding}, ACM Transactions on Mathematical Software
  (TOMS), 33 (2007), p.~13.

\bibitem{GlaserLiuRokhlin2007}
{\scshape A.~Glaser, X.~Liu, and V.~Rokhlin}, {\em A fast algorithm for the
  calculation of the roots of special functions}, SIAM Journal on Scientific
  Computing, 29 (2007), pp.~1420--1438.

\bibitem{Hobson1931}
{\scshape E.~W. Hobson}, {\em The theory of spherical and ellipsoidal
  harmonics}, CUP Archive, 1931.

\bibitem{KlocknerBarnettGreengardONeil2013}
{\scshape A.~Kl{\"o}ckner, A.~Barnett, L.~Greengard, and M.~OʼNeil}, {\em
  Quadrature by expansion: A new method for the evaluation of layer
  potentials}, Journal of Computational Physics, 252 (2013), pp.~332--349.

\bibitem{Klotz2017}
{\scshape T.~Klotz}, {\em Accurate evaluation of ellipsoidal harmonics using
  tanh-sinh quadrature}, Master's thesis, Computational and Applied
  Mathematics, Rice University, Houston, TX, 2017.

\bibitem{Miller1977}
{\scshape W.~Miller, Jr}, {\em Symmetry and separation of variables},
  Addison-Wesley Publishing Co., Inc., Reading, MA, 1977.

\bibitem{Mori2005}
{\scshape M.~Mori}, {\em Discovery of the double exponential transformation and
  its developments}, Publications of the Research Institute for Mathematical
  Sciences, 41 (2005), pp.~897--935.

\bibitem{Romain2001}
{\scshape G.~Romain and B.~Jean-Pierre}, Celestial Mechanics and Dynamical
  Astronomy, 79 (2001), pp.~235--275.

\bibitem{Sugihara1997}
{\scshape M.~Sugihara}, {\em Optimality of the double exponential formula -
  functional analysis approach -}, Numerische Mathematik, 75 (1997),
  pp.~379--395.

\bibitem{TakahasiMori1974}
{\scshape H.~Takahasi and M.~Mori}, {\em Double exponential formulas for
  numerical integration}, Publications of the Research Institute for
  Mathematical Sciences, 9 (1974), pp.~721--741.

\bibitem{Tanaka2008}
{\scshape K.~Tanaka, M.~Sugihara, K.~Murota, and M.~Mori}, {\em Function
  classes for double exponential integration formulas}, Numerische Mathematik,
  111 (2008), pp.~631--655.

\end{thebibliography}
